\documentclass[12pt]{amsart}

 \usepackage{amssymb}

 \usepackage{url}

\newtheorem{theorem}{Theorem}[section]

\newtheorem{lemma}[theorem]{Lemma}

\theoremstyle{remark}

\theoremstyle{definition}

\numberwithin{equation}{section}
\numberwithin{theorem}{section}

\theoremstyle{plain}
\newtheorem*{midpoint}{Midpoint Rule}
\newtheorem*{trapezoidal}{Trapezoidal Rule}
\newtheorem*{simpsons}{Simpson's Rule}

\begin{document}

%%%%%%%%%%%%%%%%%%%%%%%%%
% Subject classification 
%%%%%%%%%%%%%%%%%%%%%%%%%

% Provide an AMS subject classification with one or two primary classification 
% numbers and, optionally, one or more secondary classification numbers. 
% Use the following format:  "Primary 42B25. Secondary 42B60, 20E26"


\date{Preprint December 21, 2011. To appear in {\it Atlantic Electronic
Journal of Mathematics}}
%%%%%%%%%
% Title
%%%%%%%%%

% Title, in lower case, with no explicit linebreaks (\\).  If the title
% is too long to be used as a running head, add a short version of the
% title in brackets, as in \title[shorttitle]{fulltitle}.

\title{Simple derivation of basic quadrature formulas}

%%%%%%%%%%%%%%%%%%%%%%%%%%%%%%
% Author names and addresses 
%%%%%%%%%%%%%%%%%%%%%%%%%%%%%%

% Provide one separate \author{...} \address{...} \email{....} entry for each
% author, i.e., do not combine multiple authors.  Separate address lines by double
% slashes.  Do not attach footnotes to author  names. (For acknowledgements use
% the "\thanks" construct below.)
%
\author{Erik Talvila}
\address{Department of Mathematics \& Statistics\\
University of the Fraser Valley\\
Abbotsford, BC Canada V2S 7M8}
\email{Erik.Talvila@ufv.ca}

\author{Matthew Wiersma}
\address{Department of Pure Mathematics\\
University of Waterloo\\
Waterloo, ON Canada N2L 3G1}
\email{mwiersma@uwaterloo.ca}
\thanks{The first author was supported by a Discovery Grant,
the second author was supported by an Undergraduate Student Research Award; 
both from the
Natural Sciences and Engineering Research Council of Canada.
This paper was written while the
first author was on leave and visiting the Department of Mathematical and 
Statistical
Sciences, University of Alberta, and while the second author was a student at 
University of the Fraser Valley.}

%%%%%%%%%%%%%%%%%%%%
% Acknowledgements
%%%%%%%%%%%%%%%%%%%

% Use \thanks for acknowledgements as footnotes to the title page.  
% (Note that footnotes inside \author or \title macros are not
% allowed.)
%
% In case of multiple author papers, phrase the acknowledgement to 
% say "The first author was supported by ...  The second author was
% supported by ..."

%%%%%%%%%%%%%
% Abstract 
%%%%%%%%%%%%%
%
% Abstracts should not contain macros (so that they can be processed independently
% of the paper.) Avoid displayed math and references in the abstract.

\begin{abstract}
Simple proofs of the midpoint, trapezoidal and Simpson's rules are
proved for numerical integration on a compact interval.  The integrand
is assumed to be twice continuously differentiable for the midpoint
and trapezoidal rules, and to be four times continuously differentiable
for Simpson's rule.  Errors are estimated in terms of the uniform
norm of second or fourth derivatives of the integrand.  
The proof uses only integration by parts, applied to
the second or fourth derivative of the integrand, multiplied by an
appropriate polynomial or piecewise polynomial function.  A corrected
trapezoidal rule that includes the first derivative of the integrand at the
endpoints of the integration interval is also proved in this manner, 
the coefficient in the
error estimate
being smaller than for the midpoint and trapezoidal rules.  
The proofs are suitable
for presentation in a calculus or elementary numerical analysis class.
Several student projects are suggested.
\end{abstract}

\maketitle

% New definition of square root:
% it renames \sqrt as \oldsqrt
\let\oldsqrt\sqrt
% it defines the new \sqrt in terms of the old one
\def\sqrt{\mathpalette\DHLhksqrt}
\def\DHLhksqrt#1#2{%
\setbox0=\hbox{$#1\oldsqrt{#2\,}$}\dimen0=\ht0
\advance\dimen0-0.2\ht0
\setbox2=\hbox{\vrule height\ht0 depth -\dimen0}%
{\box0\lower0.4pt\box2}}

\newcounter{projects}
\newcommand{\fn}{\!:\!}
\newcommand{\bv}{{\mathcal BV}}
\newcommand{\be}{\begin{equation}}
\newcommand{\ee}{\end{equation}}
\providecommand{\abs}[1]{\lvert#1\rvert}
\providecommand{\norm}[1]{\lVert#1\rVert}
\newcommand{\intab}{\int_a^b}
\newcommand{\polyk}{{\mathcal P}_k}
\newcommand{\poly}[1]{{\mathcal P}_{#1}}
\newcommand{\R}{{\mathbb R}}
\section{Introduction}
Virtually every calculus text contains a section on numerical integration.  
Typically,
the midpoint, trapezoidal and Simpson's rules are given.  Derivation of
these quadrature formulas are usually presented, often in a graphical manner, 
but most texts
shy away from giving proofs of the
error estimates.  For example, according to \cite{peterson}, the book
{\it Calculus}, by James Stewart,
currently outsells all other calculus
texts combined in North America.  
This astonishingly popular middle brow book gives error formulas 
for the midpoint, trapezoidal and 
Simpson's rules but provides no proofs \cite{stewart}.
In this paper we give simple proofs of these three
basic quadrature rules and also a modified trapezoidal rule that includes first
derivative terms and has a smaller error estimate than the usual midpoint and
trapezoidal rules (Theorem~\ref{theoremcorrectedtrapezoidal}).  The
proofs are based on integration by parts of $\intab f''(x)p(x)\,dx$
or $\intab f^{(4)}(x)p(x)\,dx$,
where $\intab f(x)\,dx$ is the integral the rule applies to and
$p$ is a polynomial or piecewise polynomial.  Some elementary optimisation
is also required.  The proofs of these four rules are all easy enough for
a standard calculus course.

This paper will also be useful for a numerical analysis
class.  The proofs are self-contained except for an elementary lemma
on polynomials (Lemma~\ref{lemma}).  We feel they are much simpler than methods
usually employed in such courses.  These often involve developing
the theory of polynomial interpolation or special versions of
the mean value theorem.  Our proofs are constructive.  For the
midpoint and trapezoidal rules they
begin with $\intab f''(x)p(x)\,dx$, where $p$ is a generic, monic quadratic
or piecewise quadratic function.  For Simpson's rule we begin
with $\intab f^{(4)}(x)p(x)\,dx$, where $p$ is a monic,
piecewise quartic function. 
After integration by parts it is
clear what $p$ has to be.  For example,
upon integrating by parts, one easily sees that the midpoint rule arises
when 
$p(x)=(x-a)^2$ for $a\leq x\leq c$
and $p(x)=(x-b)^2$ for $c\leq x\leq b$.  See Section~\ref{sectionmidpoint}.
This makes it easy to produce new quadrature formulas.  Our corrected
trapezoidal rule, Theorem~\ref{theoremcorrectedtrapezoidal}, is
constructed so that the error is proportional to $(b-a)^3\norm{f''}_\infty$
and the constant of proportionality is the smallest possible.  The
method we use appears in \cite{hildebrand} and
\cite{cruzuribeJIPAM}. 
Both these sources 
give references to earlier practitioners of this
method, such as Peano and von Mises.

In Section~\ref{sectionclassroom}, we list
a number of exercises, problems and projects.  Some are at the
calculus level but most are at the level of an undergraduate
numerical analysis class.

We will consider numerical approximation of $\int_a^b f(x)\,dx$, under the
assumption that $f$ and its derivatives can be computed.  For the midpoint
and trapezoidal rules we assume $f\in C^2([a,b])$ ($f$ and its derivatives
to order $2$ are continuous on interval $[a,b]$).  For Simpson's rule
we assume $f\in C^4([a,b])$.  Error estimates will be obtained from 
integrals of the form $\intab f^{(m)}(x)p(x)\,dx$ where $p$ is a polynomial
or piecewise polynomial and $m$ is $2$ or $4$.  
Thus, all integrals that appear can be considered
as Riemann integrals.  In Section~\ref{sectionclassroom}, projects 
\ref{relaxing}, \ref{pnorms} and \ref{hk} discuss how assumptions
on $f$ can be weakened somewhat and then errors can be given in terms of
Lebesgue or Henstock--Kurzweil integrals.

The usual midpoint, trapezoidal and Simpson's rules are as follows.
Let $n$ be a natural number.  For $0\leq i\leq n$ define $x_i=a+
(b-a)i/n$.  The midpoint of interval $[x_{i-1},x_i]$ is 
$y_i=a+(b-a)(2i-1)/(2n)$.  The symbol $\norm{f}_\infty$ is the
uniform norm of $f$ and denotes the
supremum of $\abs{f(x)}$ for $x\in[a,b]$.  If $f$ is continuous then
this is the maximum of $\abs{f(x)}$. 
\begin{midpoint}
Let $f\in C^2([a,b])$.  Write 
\begin{equation}
\intab f(x)\,dx=(b-a)f((a+b)/2) + E^M(f).
\end{equation}
Then $|E^M(f)|\leq (b-a)^3\norm{f''}_\infty/24$.
The composite midpoint rule is
$$
\intab f(x)\,dx=\frac{b-a}{n}\sum_{i=1}^n f(y_i) + E^M_n(f),
\text{ with } \abs{E^M_n(f)}\leq\frac{(b-a)^3\norm{f''}_\infty}{24n^2}.
$$
\end{midpoint}
\begin{trapezoidal}
Let $f\in C^2([a,b])$. Write
\begin{equation}
\intab f(x)\,dx=\frac{b-a}{2}\left[f(a)+f(b)\right]
+E^{T}(f).
\end{equation}
Then $|E^{T}(f)|\leq (b-a)^3\norm{f''}_\infty/12$.
The composite trapezoidal rule is
$$
\intab f(x)\,dx=\frac{b-a}{2n}\left[f(a)+2\sum_{i=1}^{n-1}f(x_i) +f(b)\right] + 
E^T_n(f),
\text{ with } \abs{E^T_n(f)}\leq\frac{(b-a)^3\norm{f''}_\infty}{12n^2}.
$$
\end{trapezoidal}
\begin{simpsons}
Let $f\in C^4([a,b])$.  Write
\begin{equation}
\intab f(x)\,dx=\frac{b-a}{6}\left[f(a)+4f((a+b)/2)+f(b)\right]
+
E^S(f).\label{simpsonsrule}
\end{equation}
Then $|E^S(f)|\leq (b-a)^5\norm{f^{(4)}}_\infty/2880$.
Let $n$ be even.  The composite Simpson's rule is
\begin{equation}
\intab f(x)\,dx=\frac{b-a}{3n}\left[f(a)+2\sum_{i=1}^{n/2-1}f(x_{2i})
+4\sum_{i=1}^{n/2}f(x_{2i-1})+f(b)\right]
+
E^S_n(f),
\end{equation}
with $\abs{E^S_n(f)}\leq(b-a)^5\norm{f^{(4)}}_\infty/(180n^4)$.
\end{simpsons}
Many authors give the error for the trapezoidal rule as
$E^T_n(f)=(b-a)^3f''(\xi)/(12n^2)$,
where $\xi$ is some point
in $[a,b]$.  There are similar forms for the other rules.
We don't find these any more useful than the uniform norm estimates.
Unless we know something about $f$ beyond continuity of its
derivatives, it is impossible to say what $\xi$ is.

Note that the approximation in Simpson's rule can be written
$$
\intab f(x)\,dx\doteq\frac{b-a}{3n}\left[f(a)+4f(x_1)+2f(x_2)+\cdots
+2f(x_{n-2})+4f(x_{n-1})+f(b)\right].
$$
Proofs of these three rules are given in Sections~\ref{sectionmidpoint},
\ref{sectiontrapezoidal} and \ref{sectionsimpsons}, respectively.

The literature on these formulas is vast.  Here is a sample of some
of the different methods of proof that have been published in 
calculus texts.
There are proofs based on the mean value theorem
and Rolle's theorem \cite{lightstone}, and polynomial interpolation
\cite{apostol}.
Several authors produce a somewhat mystical auxiliary function and
employ the mean value theorem or intermediate value theorem with
integration by parts.
For example, \cite{kaplan}, \cite{olmsted}.
All of the methods listed above appear in several sources.

There are many elementary journal articles
that treat numerical integration.  For
a geometrical version of the midpoint rule, see  Hammer \cite{hammer}.
Cruz-Uribe and  Neugebauer \cite{cruzuribemathmag} give a basic
proof of the trapezoidal rule using integration by parts.
Rozema \cite{rozema} shows how to estimate the error for the trapezoidal rule,
Simpson's rule and various versions of these rule that are
corrected with derivative terms.
Hart \cite{hart} also considers corrected versions of the
trapezoidal rule.
Talman
\cite{talman}
proves Simpson's rule by using an extended
version of the mean value theorem for integrals.  
For other commentary
on Simpson's rule, see \cite{shipp} and \cite{velleman}.

For a numerical analysis course, integration of polynomial interpolation
approximations is frequently used.  See \cite{conteboor}.
See \cite{hildebrand} for proofs based on the difference calculus.
For Taylor series, \cite{thompson}.  The elementary textbook 
\cite{burdenfaires} uses a rather complicated method with Taylor
series and a weighted mean value theorem for integrals.
For more sophisticated audiences, there are proofs based on the
Euler--Maclaurin summation formula
and the Peano kernel.
See \cite{davisrabinowitz} and \cite{krylov}.
General references for numerical integration are \cite{davisrabinowitz},
\cite{engels},
\cite{krommer},
\cite{kythe},
\cite{squire},
\cite{ueberhuber} and
\cite{zwillinger}.
Several other methods can be found here.

\section{Trapezoidal rule}\label{sectiontrapezoidal}
For all of the quadrature formulas we derive, the error is estimated
from the integral $\intab f^{(m)}(x)p(x)\,dx$, where $p$ is a suitable 
polynomial or piecewise polynomial function.  

We first consider the trapezoidal rule.  The estimate is
then $\intab f(x)\,dx\doteq (b-a)[f(a)+f(b)]/2$.  
\begin{proof}
Write $p(x)=(x-\alpha)^2+\beta$,
where the constants $\alpha$ and $\beta$ are to be determined.  Assume
$f\in C^2([a,b])$.  Integrate by parts to get
\begin{eqnarray*}
\intab f''(x)p(x)\,dx & = & f'(b)p(b)-f'(a)p(a)-\intab f'(x)p'(x)\,dx\\
 & = &  f'(b)p(b)-f'(a)p(a)-f(b)p'(b)+f(a)p'(a)+\intab f(x)p''(x)\,dx.\\
\end{eqnarray*}
Since $p''=2$ we can solve for $\intab f(x)\,dx$,
\begin{equation}
\intab f(x)\,dx=\frac{1}{2}\left[-f(a)p'(a)+f(b)p'(b)+f'(a)p(a)-f'(b)p(b)\right]
+E(f),\label{basic}
\end{equation}
where $E(f)=\frac{1}{2}\intab f''(x)p(x)\,dx$.  To get the trapezoidal rule
we require $p(a)=p(b)=0$ and $-p'(a)=p'(b)=b-a$.  (Since we want the
trapezoidal rule for all such $f$, the four variables $f(a)$, $f(b)$,
$f'(a)$ and $f'(b)$ are linearly independent.)  The solution of this
overdetermined system is $\alpha=(a+b)/2$
and $\beta=-(b-a)^2/4$.  The required quadratic is then
$p(x)=[x-(a+b)/2]^2-(b-a)^2/4$.  Now we can estimate the error by
\begin{equation}
|E(f)|  \leq  \frac{1}{2}\intab|f''(x)p(x)|\,dx
 \leq  \frac{\norm{f''}_\infty}{2}\intab|p(x)|\,dx.\label{holder}
\end{equation}
To evaluate the last integral, let $h=(b-a)/2$ and note that
$$
\intab|p(x)|\,dx=\int_{-h}^h|x^2-h^2|\,dx=2\int_0^h(h^2-x^2)\,dx=4h^3/3.
$$
We then get $|E(f)|\leq (b-a)^3\norm{f''}_\infty/12$.  

Now let $n\geq 2$
and  use this
estimate on each interval $[x_{i-1},x_i]$ for $1\leq i\leq n$.  Let
$y_i$ be the midpoint of $[x_{i-1},x_i]$. We define
the piecewise quadratic function $P\fn[a,b]\to\R$ by 
$P(x)=(x-y_i)^2-(b-a)^2/(4n^2)$ if 
$x\in[x_{i-1},x_i]$ for some $1\leq i\leq n$.
Now we have $P$ continuous on $[a,b]$
with $P(x_{i-1})=P(x_i)=0$, $P'(x_i-)=(b-a)/n$ and $P'(x_i+)=-(b-a)/n$.  
For the
composite rule, \eqref{basic} gives
\begin{eqnarray*}
\intab f(x)\,dx & = & \sum_{i=1}^n\int_{x_{i-1}}^{x_i}f(x)\,dx
 \doteq \frac{(b-a)}{2n}\sum_{i=1}^n\left[f(x_{i-1})
+f(x_i)\right]\\
 & = & \frac{(b-a)}{2n}\left\{f(a)+2[f(x_1)+f(x_2)+\cdots+f(x_{n-1})]+f(b)
\right\}.
\end{eqnarray*}
Let $\Delta x=(b-a)/n$.  The error is
\begin{align*}
&|E^T_n(f)|  =  \frac{1}{2}
\left|\sum_{i=1}^n\int_{x_{i-1}}^{x_i}f''(x)P(x)\,dx\right|
\leq   \frac{\norm{f''}_\infty}{2}\sum_{i=1}^n \int_{x_{i-1}}^{x_i}\left|
(x-y_i)^2- \left(\frac{\Delta x}{2}\right)^2\right|\,dx\\
& =\norm{f''}_\infty\sum_{i=1}^n\int_0^{\Delta x/2}\left[
\left(\frac{\Delta x}{2}\right)^2
-x^2\right]dx
 = \norm{f''}_\infty\sum_{i=1}^n\frac{2}{3}
\left(\frac{\Delta x}{2}\right)^3
  =   \frac{(b-a)^3\norm{f''}_\infty}{12n^2}.
\end{align*}
\end{proof}
We consider this to be a completely elementary derivation of the trapezoidal
rule.  The method is perfectly suitable for presenting in a calculus class or
numerical 
analysis class.

Notice that $p(x)=(x-a)(x-b)$ so it is not necessary for
$f'$ to be continuous, provided $f'(x)(x-a)$ and $f'(x)(x-b)$ have
limits as $x\to a^+$ and $x\to b^-$, respectively.  In this case,
$f''$ will not be bounded so different methods will be needed to
estimate $\intab f''(x)p(x)\,dx$.  See projects~\ref{pnorms} and
\ref{unbounded} in Section~\ref{sectionclassroom}.

\section{Corrected trapezoidal rule}\label{sectioncorrectedtrap}
Quadrature rules are often constructed so that they are exact for
polynomials of a certain degree.  For example, see \cite[\S5.10]{hildebrand}.
Here we do something different.  We will minimise the coefficient
in the error estimate.
In \eqref{holder}, we have 
$|\intab f''(x)p(x)\,dx|
\leq  \norm{f''}_\infty\intab|p(x)|\,dx$.
This is a version of the H\"older inequality and it is a standard result
of functional analysis that this
is the best possible estimate over all such $f$ and $p$.  See, for example,
\cite[p.~184]{folland}.  For more on this point, 
see project~\ref{holderequality}
in Section~\ref{sectionclassroom}.  
This begs the question: What values of $\alpha$
and $\beta$ will minimise $\intab|p(x)|\,dx$?  One of the requirements that
we obtained the trapezoidal rule in the above calculation was that
the coefficients of $f'(a)$ and $f'(b)$ vanish in \eqref{basic}.  When
we choose $\alpha$ and $\beta$ to minimise $\intab|p(x)|\,dx$ the
resulting quadrature formula will have derivatives of $f$.  But who cares?
If we are assuming we can estimate $f''$ then surely we can include
first derivative terms.

We first need a lemma about polynomials that can 
minimise $\intab|p(x)|\,dx$.
\begin{lemma}\label{lemma}
Fix $k\geq 1$. Let $\polyk$ be the monic polynomials of degree $k$, with 
real coefficients. Let $I(p)=\intab|p(x)|\,dx$.
If $p\in\polyk$ minimises $I$ then $p$ 
has $k$ real roots in $[a,b]$, counting multiplicities.
\end{lemma}
\begin{proof}
If $k=1$ evaluation of $\intab|x-c|\,dx$ shows the minimum occurs when
$c=(a+b)/2$.  This can also be seen graphically.

Now assume $k\geq 2$. If $I$ is minimised by $p\in\polyk$ and $p$
has a root that is not real then write $p(x)=[(x-c)^2+d^2]q(x)$ where
$c,d\in\R$, $d>0$ and $q\in\poly{k-2}$.  Then
$$
\intab|p(x)|\,dx  =  \intab(x-c)^2|q(x)|\,dx+d^2\intab|q(x)|\,dx
>  \intab(x-c)^2|q(x)|\,dx,
$$
contradicting the assumption that $p$ minimises $I$.  A minimising
polynomial then has $k$ real roots, counting multiplicities.

Now suppose $p\in\polyk$ minimises $I$ and $p(a-c)=0$ for some $c>0$.
Then $p(x)=(x-a+c)q(x)$ for some $q\in\poly{k-1}$.  And,
$$
\intab|p(x)|\,dx  =  \intab(x-a)|q(x)|\,dx+c\intab|q(x)|\,dx
 > \intab(x-a)|q(x)|\,dx,
$$
contradicting the assumption that $p$ minimises $I$.  Hence, $p$ cannot
have any roots that are less than $a$. A similar argument shows $p$ cannot
have roots greater than $b$. 
\end{proof}

Now we can prove the corrected trapezoidal rule.
\begin{theorem}[Corrected trapezoidal rule]\label{theoremcorrectedtrapezoidal}
Let $f\in C^2([a,b])$.  Write
\begin{equation}
\intab f(x)\,dx=\frac{b-a}{2}\left[f(a)+f(b)\right]+\frac{3(b-a)^2}{32}\left[
f'(a)-f'(b)\right]
+E^{CT}(f).\label{correctedtrapezoidal}
\end{equation}
Then $|E^{CT}(f)|\leq (b-a)^3\norm{f''}_\infty/32$.
The composite corrected trapezoidal rule is
$$
\intab f(x)\,dx=\frac{b-a}{2n}\left[f(a)+2\sum_{i=1}^{n-1}f(x_i) +f(b)\right]
+\frac{3(b-a)^2}{32n^2}\left[f'(a)-f'(b)\right] +E^{CT}_n(f),
$$
with $\abs{E^{CT}_n(f)}\leq(b-a)^3\norm{f''}_\infty/(32n^2)$.
\end{theorem}
\begin{proof}
As in the proof of the trapezoidal rule in Section~\ref{sectiontrapezoidal},
we are led to  $\intab|p(x)|\,dx$ for $p$ a polynomial in $\poly{2}$.
But this time, we choose $p$ to minimise this integral.
Due to the lemma, we can write $p(x)=(x-\alpha)^2 - \gamma^2$ where
$a\leq \alpha-\gamma\leq \alpha+\gamma\leq b$. Then
$p$ has zeros at $\alpha\pm\gamma$, which are in $[a,b]$.
Let $q(\alpha,\gamma)=\intab|(x-\alpha)^2 - \gamma^2|\,dx=
\int_{a-\alpha}^{b-\alpha}|x^2-\gamma^2|\,dx$.  This must be
minimised over the triangular region $Q=\{(x,y)\in\R^2\mid a\leq x\leq b,
0\leq y\leq \min(x-a,b-x)\}$.  Differentiating the integral with respect
to $\alpha$, we have
\begin{eqnarray*}
\partial q(\alpha,\gamma)/\partial\alpha & = & 
|(a-\alpha)^2-\gamma^2|-|(b-\alpha)^2-\gamma^2|\\
 & = & a^2-2a\alpha-b^2+2b\alpha\\
 &  & \left\{\begin{array}{ll}
<0,\quad\text{when } a\leq \alpha<(a+b)/2\\ 
=0,\quad\text{when } \alpha=(a+b)/2\\
>0,\quad\text{when } (a+b)/2< \alpha\leq b.
\end{array}
\right.
\end{eqnarray*}
Hence, for each allowed $\gamma$ the minimum of $q$ in $Q$ occurs at
$\alpha=(a+b)/2$.  Now let 
$r(\gamma)=q((a+b)/2,\gamma)=2\int_0^h|x^2-\gamma^2|\,dx$, where
$h=(b-a)/2$.   Differentiating under the integral sign, we have
\begin{eqnarray*}
r'(\gamma) & = & -4\gamma\int_0^h{\rm sgn}(x^2-\gamma^2)\,dx
 =  -4\gamma\left(-\int_0^\gamma dx+\int_\gamma^h dx\right)\\
 & = & 8\gamma(\gamma-h/2)
\quad \left\{\begin{array}{ll}
<0,\quad\text{when } 0< \gamma<h/2\\ 
=0,\quad\text{when } \gamma= 0 \text{ or } h/2\\
>0,\quad\text{when } h/2< \gamma\leq h.
\end{array}
\right.
\end{eqnarray*}
Hence, the minimum of $r$ occurs at $\gamma=h/2=(b-a)/4$.  Now evaluate
\begin{eqnarray*}
q((a+b)/2,h/2) & = & 2\int_{0}^h|x^2-h^2/4|\,dx
 =  2h^3\left(\int_0^{1/2}(1/4-x^2)\,dx+\int_{1/2}^1(x^2-1/4)\,dx\right)\\
 & = & h^3/2=(b-a)^3/16.
\end{eqnarray*}
The minimising polynomial is then $p(x)=(x-(a+b)/2)^2-(b-a)^2/16$.
Using \eqref{basic} we have
\begin{equation}
\intab f(x)\,dx=\frac{b-a}{2}\left[f(a)+f(b)\right]+\frac{3(b-a)^2}{32}\left[
f'(a)-f'(b)\right]
+E^{CT}(f),
\end{equation}
where $|E^{CT}(f)|\leq (b-a)^3\norm{f''}_\infty/32$.

For the composite corrected trapezoidal rule, apply the above rule on
each interval $[x_{i-1},x_i]$ for $1\leq i\leq n$.  This gives
\begin{eqnarray}
\intab f(x)\,dx & \doteq &\frac{(b-a)}{2n}\sum_{i=1}^n\left[f(x_{i-1})
+f(x_i)\right]+\frac{3(b-a)^2}{32n^2}\sum_{i=1}^n\left[f'(x_{i-1})
-f'(x_i)\right]\notag\\
 & = & \frac{(b-a)}{2n}\left\{f(a)+2[f(x_1)+f(x_2)+\cdots+f(x_{n-1})]+f(b)
\right\}\notag\\
 & & \quad+\frac{3(b-a)^2}{32n^2}\left[f'(a)-f'(b)\right].\label{telescope}
\end{eqnarray}
Let $\alpha_i=(x_{i-1}+x_i)/2=y_i$ and $\gamma_i=(x_i-x_{i-1})/4=(b-a)/(4n)$.
The error estimate is
\begin{eqnarray*}
|E^{CT}(f)|\leq \frac{\norm{f''}_\infty}
{2}\sum_{i=1}^n\int_{x_{i-1}}^{x_i}|(x-\alpha_i)^2-\gamma_i^2|\,dx
\leq\frac{(b-a)^3\norm{f''}_\infty}{32n^2}.
\end{eqnarray*}
\end{proof}
In a one-variable calculus class, the minimisation problem can be done as
above but without using partial derivative notation.  Differentiating
under the integral sign with respect to $\alpha$ and $\gamma$ is justified with the Lebesgue dominated
convergence theorem since the derivative of the integrand
exists except at one point.
To avoid higher integration theory, it is easy enough to evaluate
$\intab|(x-\alpha)^2 - \gamma^2|\,dx$ before differentiating with
respect to $\alpha$ and $\gamma$.  But, as pointed out in 
project~\ref{pnorms} of Section~\ref{sectionclassroom}, 
the method used in the proof is useful for minimising
with respect to the $p$-norm of $f$.

Notice that in the composite rule the sum of derivative terms telescopes.
This means that in \eqref{telescope} only $f'(a)$ and $f'(b)$ appear.
The composite trapezoidal, midpoint and corrected trapezoidal rule all have an
error term proportional to $(b-a)^3\norm{f''}_\infty/n^2$.  The
constant of proportionality is $1/12$, $1/24$ and $1/32$, respectively.
So with the composite corrected trapezoidal rule we have a smaller 
error estimate 
but are only required to add the additional two terms
$3(b-a)^2[f'(a)-f'(b)]/(32n^2)$ to the composite trapezoidal rule.
If $n$ is reasonably large this is a negligible
amount of additional work.  If $f'$ can be computed at $a$ and $b$
this becomes an attractive quadrature rule.

The corrected trapezoidal rule given in 
Theorem~\ref{theoremcorrectedtrapezoidal}
is not the usual one that has traditionally appeared
in the literature.  For example, in
Conte and De Boor \cite{conteboor}, Davis and Rabinowitz
\cite{davisrabinowitz},
Dragomir, et al
\cite{dragomirceronesofo},
Pe\v{c}ari\`{c} and Ujevi\`{c} \cite{pecaricujevic},
and  Squire \cite{squire}, the coefficient
is $1/12$ in place of our $3/32$ in \eqref{correctedtrapezoidal}.
The error estimate $(b-a)^5\norm{f^{(4)}}_\infty/720$ is obtained by
polynomial interpolation by Conte and De Boor in \cite{conteboor} 
and with a two-point Taylor
expansion  by Davis and Rabinowitz in 
\cite{davisrabinowitz}.
Dragomir, et al \cite{dragomirceronesofo}, use Gr\"{u}ss's inequality.  
In their Lemma~2,
the error is given
with $\norm{f''}_\infty$ replaced by
$\sup_{[a,b]}f''-\inf_{[a,b]}f''$.
Pe\v{c}ari\`{c} and Ujevi\`{c} \cite{pecaricujevic} give the
error estimate as $\sqrt{3}(b-a)^3\norm{f''}_\infty/54$
in their equation (3.3).  This also appears in  
Dedi\`{c}, et al \cite{dedic}.
Cerone and Dragomir \cite{ceronedragomirtrap} 
have coefficient
$1/8$ in their equation ($3.64$) in place of our $3/32$ in 
\eqref{correctedtrapezoidal}.
Their error estimate is $(b-a)^3\norm{f''}_\infty/24$, obtained
with integration by parts.
Squire \cite{squire} gives a number of 
rules that use
derivatives but does not provide any error estimates.  It is shown in
\cite{talvilawiersma} that 
the coefficient $1/32$ in Theorem~\ref{theoremcorrectedtrapezoidal}
is the best possible.

\section{Midpoint rule}\label{sectionmidpoint}
Notice that with the composite trapezoidal rule, values of $f$ were 
brought forth  at
discontinuities in the
derivative of $p$.  For the midpoint rule we will define $p$  so that there is
a discontinuity in $p'$ at the midpoint $c=(a+b)/2$.  Assume $p$ is piecewise
monic quadratic so that it is continuous
on $[a,b]$ with $p'$ continuous
on $[a,c)$ and on $(c,b]$.
\begin{proof}
Integrating by parts twice,
\begin{align*}
&\intab f''(x)p(x)\,dx  =  \int_a^c f''(x)p(x)\,dx + \int_c^b f''(x)p(x)\,dx\\
& =  -f'(a)p(a)+f(a)p'(a)+f'(b)p(b)-f(b)p'(b)-f(c)[p'(c-)-p'(c+)]
+2\intab f(x)\,dx.
\end{align*}
For the midpoint rule we require $p(a)=p'(a)=p(b)=p'(b)=0$ and
$p'(c-)-p'(c+)=2(b-a)$.  This gives
$$
p(x)=\left\{\begin{array}{cl}
(x-a)^2, & a\leq x\leq c\\
(x-b)^2, & c\leq x\leq b.\\
\end{array}
\right.
$$
The error satisfies
$$
|E^M(f)|\leq \frac{\norm{f''}_\infty}{2}\left(\int_a^c(x-a)^2\,dx+
\int_c^b(x-b)^2\,dx\right)=\frac{\norm{f''}_\infty(b-a)^3}{24}.
$$

The composite rule follows as with the composite trapezoidal rule.
Note that  $p$ and $p'$ vanish at $a$ and $b$.  Define
$P(x)=(x-x_{i-1})^2$ for $x_{i-1}\leq x\leq y_i$ and
$P(x)=(x-x_i)^2$ for $y_{i}< x<x_i$ for $1\leq i\leq n$.
Then $P$ and $P'$ have discontinuities only at the midpoints $y_i$.
Integrating by parts $\intab f''(x)P(x)\,dx$ then gives the
composite rule.
\end{proof}

Notice that $p(x)=(x-a)^2$ for $a\leq x\leq c$
and $p(x)=(x-b)^2$ for $c\leq x\leq b$.
Hence, it is not necessary for $f$ or
$f'$ to be continuous, provided $f'(x)(x-a)^2$ and $f(x)(x-a)$ have
limits as $x\to a^+$. Similarly, as $x\to b^-$.  In this case,
$f''$ will not be bounded so different methods will be needed to
estimate $\intab f''(x)p(x)\,dx$.  See projects~\ref{pnorms} and
\ref{unbounded} in Section~\ref{sectionclassroom}.

Various versions of the midpoint rule are given in
\cite{ceronedragomirmid}.

\section{Simpson's rule}\label{sectionsimpsons}
In Simpson's rule there are function evaluations at endpoints
$a$, $b$ and at midpoint $c$.  As we saw with the midpoint rule,
when we integrate $\intab f^{(4)}(x)p(x)\,dx$, discontinuities
in $p$ and its derivatives at $c$ lead to evaluations of $f$ and its
derivatives at $c$. 
Assume that $p$ is a monic quartic
polynomial on $[a,c)$ and on $(c,b]$.  As we will now see,
the requirement that
$p\in C^2([a,b])$ determines the coefficients of $f(a)$, $f(b)$ and
$f(c)$ in Simpson's rule.  A brief explanation of this phenomenon
appears in \cite{rennie}.  It is similar to the construction of the
Green's function for ordinary differential equations.

\begin{proof}
Integrate by parts four times to get
\begin{align}
&\intab f^{(4)}(x)p(x)\,dx  =  
-f'''(a)p(a)+f'''(c)\left[p(c-)-p(c+)\right] +f'''(b)p(b)
+f''(a)p'(a)\notag\\
&-f''(c)\left[p'(c-)-p'(c+)\right] -f''(b)p'(b)
-f'(a)p''(a)+f'(c)\left[p''(c-)-p''(c+)\right]\label{simpson}\\
& +f'(b)p''(b)
+f(a)p'''(a)-f(c)\left[p'''(c-)-p'''(c+)\right] -f(b)p'''(b)
+24\intab f(x)\,dx.\notag
\end{align}
For our quadrature rule to have no evaluations of derivatives of $f$ we need
$p(a)=p'(a)=p''(a) =p(b)=p'(b)=p''(b)=0$.   
This means there are constants $d_1$ and $d_2$ such that
$$
p(x)=\left\{\begin{array}{cl}
(x-a)^3(x+d_1), & a\leq x\leq c\\
(x-b)^3(x+d_2), & c\leq x\leq b.\\
\end{array}
\right.
$$
Continuity of $p$ at $c$ requires $p(c-)=p(c+)$.  From this it follows that
$d_1+d_2=-(a+b)$.  The derivative of $p$ is
$$
p'(x)=\left\{\begin{array}{cl}
(x-a)^2(4x+3d_1-a), & a\leq x< c\\
(x-b)^2(4x+3d_2-b), & c< x\leq b.\\
\end{array}
\right.
$$
Continuity of $p'$ at $c$ requires $p'(c-)=p'(c+)$.  From this it follows that
$3(d_2-d_1)=b-a$.  Solving these two linear equations gives 
$d_1=-(a+2b)/3$ and $d_2=-(2a+b)/3$.  We now have
$$
p''(x)=\left\{\begin{array}{cl}
4(x-a)(3x-2a-b), & a\leq x< c\\
4(x-b)(3x-a-2b), & c< x\leq b.\\
\end{array}
\right.
$$ 
This shows that $p''(c-)=p''(c+)=(b-a)^2$.  So $p\in C^2([a,b])$.  Now,
$$
p'''(x)=\left\{\begin{array}{cl}
4(6x-5a-b), & a\leq x< c\\
4(6x-a-5b), & c< x\leq b.\\
\end{array}
\right.
$$ 
And, $p'''(a)=-4(b-a)$, $p'''(b)=4(b-a)$,  $p'''(c-)-p'''(c+)=16(b-a)$.
From \eqref{simpson} we get the required approximation
in \eqref{simpsonsrule}.  

The polynomial we are using is
$$
p(x)=\left\{\begin{array}{cl}
(x-a)^3(x-a/3-2b/3), & a\leq x\leq c\\
(x-b)^3(x-2a/3-b/3), & c\leq x\leq b.\\
\end{array}
\right.
$$
The error is then
$$
|E^S(f)|= \frac{1}{24}\left|\intab f^{(4)}(x)p(x)\,dx\right|
\leq \frac{\norm{f^{(4)}}_\infty}{24}\intab|p(x)|\,dx.
$$
Note that $a/3+2b/3-(a+b)/2=(b-a)/6>0$ and
$2a/3+b/3-(a+b)/2=(a-b)/6<0$.  Therefore,
$\intab|p(x)|\,dx=\int_a^c(x-a)^3(a/3+2b/3-x)\,dx+\int_c^b(b-x)^3
(x-2a/3-b/3)\,dx$.  The transformation $x\mapsto a+b-x$ shows these
last two integrals are equal.  Hence,
\begin{eqnarray*}
\intab|p(x)|\,dx & = & 2\int_a^c(x-a)^3(a/3+2b/3-x)\,dx\\
 & = & -2\int_a^c(x-a)^4\,dx+\frac{4(b-a)}{3}\int_a^c(x-a)^3\,dx\\
 & = & (b-a)^5/120.
\end{eqnarray*}
This gives Simpson's rule. 

For the composite rule it is traditional to take $n$ even, divide
$[a,b]$ into $n/2$ equal subintervals and apply Simpson's rule on each interval
$[x_{2i-2},x_{2i}]$ for $1\leq i\leq n/2$.  The approximation is
then 
\begin{eqnarray*}
\intab f(x)\,dx & = & \sum_{i=1}^{n/2}\int_{x_{2i-2}}^{x_{2i}}f(x)\,dx
 \doteq \frac{(b-a)}{3n}\sum_{i=1}^{n/2}\left[f(x_{2i-2}) +4f(x_{2i-1})
+f(x_{2i})\right]\\
 & = & \frac{(b-a)}{3n}\left[f(a)+2\sum_{i=1}^{n/2-1}f(x_{2i})
+4\sum_{i=1}^{n/2}f(x_{2i-1})+f(b)\right].
\end{eqnarray*}
The error is computed as with the trapezoidal rule.
\end{proof}

Notice that $p(x)=O((x-a)^3)$ as $x\to a^+$.  
Hence, it is not necessary for $f'$, $f''$ or
$f'''$ to be continuous, provided $f'''(x)(x-a)^3$, 
$f''(x)(x-a)^2$ and $f'(x)(x-a)$ have
limits as $x\to a^+$. Similarly, as $x\to b^-$.  In this case,
$f^{(4)}$ will not be bounded so different methods will be needed to
estimate $\intab f^{(4)}(x)p(x)\,dx$.  See projects~\ref{pnorms} and
\ref{unbounded} in Section~\ref{sectionclassroom}.

Liu uses integration by parts to prove a version of Simpson's rule
for which $f\in C^n([a,b])$ \cite{liu}.

\section{Classroom projects}\label{sectionclassroom}
The methods we have used to produce the midpoint rule, the trapezoidal
rule, the corrected trapezoidal rule and Simpson's rule are:
integration by parts, basic optimisation, and a simple fact about integrals of
polynomials (Lemma~\ref{lemma}).  We have not needed any of the
machinery mentioned in the Introduction that is often used in
other proofs.  This means our methods are well suited for use
by students.  We list below a number of topics that can be
investigated in the classroom.  Some are at the level of a calculus
course, others would make good assignments or projects in a
beginning numerical analysis course.  A few would be suitable
for a senior undergraduate research project or perhaps  an
M.Sc. project.\\

\refstepcounter{projects}
\label{firstorder}
\noindent
{\bf
\theprojects.
First order error estimates.}
In all of the above rules it is
assumed that $f''$ exists. What if $f\in C^1([a,b])$ but $f\notin C^2([a,b])$?
For example,
$f(x)=x^\alpha$ on $[0,1]$ if $1<\alpha<2$.  Then we could still derive
quadrature formulas by using one integration by parts on 
$\int_a^b f'(x)p(x)\,dx$.  We can get the trapezoidal rule if $p$ is
a linear function.  The error estimate is then
$(b-a)^2\norm{f'}_\infty/4$.  See 
\cite{buck} for a geometric proof or \cite{cruzuribemathmag} for an
integration by parts proof.  (The constant of
proportionality is misprinted as $1/2$ in 
\cite{cruzuribemathmag}.)
Taking $p$ to be piecewise linear produces
the midpoint rule with the same error.  The paper \cite{cruzuribeJIPAM}
gives several different types of error estimates based on $f'$ for the
trapezoidal and Simpson rules.\\

\refstepcounter{projects}
\label{midpointmodification}
\noindent
{\bf
\theprojects.
Midpoint modifications.}  In the midpoint rule, what happens
if we allow evaluation of $f$ or $f'$ at the endpoints and midpoint of
$[a,b]$?  How does the composite rule then compare with the 
trapezoidal rule and corrected trapezoidal rules?\\

\refstepcounter{projects}
\label{periodic}
\noindent
{\bf
\theprojects.
Periodic functions.}  If $f$ is periodic and we integrate over
one period, how do the quadrature formulas simplify? 
Note that for a periodic function, application of the
trapezoidal rule actually gives the corrected trapezoidal
rule.  A much deeper discussion can be found in \cite{davisrabinowitz}.\\

\refstepcounter{projects}
\label{higherorder}
\noindent
{\bf
\theprojects.
Higher order error estimates.}  If $f\in C^n([a,b])$ and $p$
is a monic polynomial of degree $k\geq n$ then integrate by parts on
$\intab f^{(n)}(x) p(x)\,dx$ to get other quadrature formulas.  If
$p$ is a piecewise polynomial then $f$ and its derivatives can be
made to be evaluated at discontinuities in the derivatives of $p$.
It is possible to make a systematic study of quadrature formulas
obtained in this manner.  In the corrected trapezoidal rule, the
quadratic polynomial that minimised $\intab|f''(x)p(x)|\,dx$
caused the $f'$ terms to telescope away \eqref{telescope}.  
This phenomenon can also
be investigated for higher degree polynomials.\\

\refstepcounter{projects}
\label{linearcombinations}
\noindent
{\bf
\theprojects.
Linear combinations.}  It is well known that Simpson's rule can
be obtained as a linear combination of trapezoidal rules or of
midpoint and trapezoidal rules.  Look
for other such relationships amongst the various rules discussed above.

In Romberg integration, one takes a linear combination of trapezoidal
rules with $n$ and $2n$.  This yields a quadrature formula with
improved error estimate.  This hierarchy is then repeated. 
See \cite{davisrabinowitz}.  Does the
integral form of the trapezoidal rule error show how to do this?  Can
this be
done with the corrected trapezoidal rule?\\

\refstepcounter{projects}
\label{finitedifferences}
\noindent
{\bf
\theprojects.
Finite differences.}
If $f$ was a special function defined by a definite integral or series
depending on
a parameter then it may not be feasible to compute $f'$.  Similarly if
$f$ was given by experimental data.  In such cases, we could
approximate derivatives by finite differences, $f'(x)\doteq [f(x)-f(x+h)]/h$
if $h$ is small.
Do this for the composite corrected trapezoidal rule and compute 
the resulting error.\\

\refstepcounter{projects}
\label{relaxing}
\noindent
{\bf
\theprojects.
Relaxing conditions on $f$.}  In the estimate 
$|\intab f''(x)p(x)\,dx|\leq \norm{f''}_\infty\int_a^b|p(x)|\,dx$ it
is not necessary that $f''$ be continuous.
If we use the Lebesgue integral,
the conditions on $f$ can be weakened
to $f'$ being absolutely continuous such that $f''$ is essentially bounded.
This is the same as $f'$ being Lipschitz continuous.  Similar remarks
apply for Simpson's rule and in \ref{firstorder} above.  Under the 
assumption that $f'$ is Lipschitz continuous, what do the error estimates
for the trapezoidal, corrected trapezoidal and midpoint rules become?
What Lipschitz condition could be used for Simpson's rule?\\

\refstepcounter{projects}
\label{pnorms}
\noindent
{\bf
\theprojects.
Using other Lebesgue norms to estimate the error.}  If $f''\in L^r([a,b])$
then the H\"older inequality gives 
$|\intab f''(x)p(x)\,dx|\leq \norm{f''}_r\norm{p}_s$, with 
$1/r+1/s=1$.  The case $r=\infty, s=1$ has already been used.  The cases
$1\leq r<\infty$ could be investigated.  The case $r=s=2$ serves as a
good warm up since the integral $\int_a^b|p(x)|^2\,dx$ can be
evaluated explicitly.  The minimising method from the proof of
Theorem~\ref{theoremcorrectedtrapezoidal} can be used.  
Similarly with Simpson's rule and 1. above.  See \cite{talvilawiersma} for
$p$-norm estimates for modified trapezoidal rules.\\

\refstepcounter{projects}
\label{holderequality}
\noindent
{\bf
\theprojects.
Equality in the corrected trapezoidal error.}
At the beginning of Section~\ref{sectioncorrectedtrap} we mentioned that
$|\intab f''(x)p(x)\,dx|
\leq  \norm{f''}_\infty\norm{p}_1$.  Show that for each quadratic $p$
there is a function $f\in C^1([a,b])$ such that $f''$ is piecewise
constant and 
$\intab|f''(x)p(x)|\,dx
=\norm{f''}_\infty\norm{p}_1$.
Show that for each $\epsilon>0$ there
is a function $g\in C^2([a,b])$ such that 
$|\intab g''(x)p(x)\,dx|
\geq  \norm{g''}_\infty\norm{p}_1-\epsilon$.\\

\refstepcounter{projects}
\label{geometric}
\noindent
{\bf
\theprojects.
Geometric proofs.}  Sketch the piecewise polynomial functions used
in derivation of all the above rules.
Can you find a geometric proof of the
choice of minimising polynomial in the corrected trapezoidal rule?
What about for minimising polynomials of $\norm{p}_s$?
For example, Derek Lacoursiere has observed that if $p$ is the
monic quadratic
that minimises
$\norm{p}_\infty$ then $p(a)=|p(c)|=p(b)$.\\

\refstepcounter{projects}
\label{nonuniform}
\noindent
{\bf
\theprojects.
Non-uniform partitions.} The composite rules are much simpler when
the partition is uniform.  But by taking non-uniform partitions we can
get smaller error estimates.
This will happen if smaller
subintervals are taken where $|f''|$ is large and larger subintervals are
allowed where
$|f''|$ is small. This could be done in a systematic way if, say, $f''$
was positive and decreasing.  This opens up the creation of adaptive 
algorithms.  See \cite[p.~160]{ueberhuber} for a meta algorithm on adaptive
integration.  A basic example of such an algorithm is given in 
\cite{burdenfaires}.  Rice \cite{rice} has estimated there ``are from 
are from $1$ to $10$ million algorithms that are potentially 
interesting and significantly different from one another".
Get cracking!\\

\refstepcounter{projects}
\label{subinterval}
\noindent
{\bf
\theprojects.
Error estimates on each subinterval.} By taking properties of $f$ into 
account it
is possible to get better error estimates.
Denote the characteristic
function of interval $[s,t]$ by $\chi_{[s,t]}(x)$ and this is $1$
if $x\in[s,t]$ and $0$, otherwise.
The estimate 
$\norm{f''\chi_{[x_{i-1},x_i]}}_\infty\leq\norm{f''}_\infty$ was used in
the proof of the trapezoidal rule.  (Can you see where?)
It is the best we can do
for generic $f$ such that $f''$ is bounded, since then the supremum
of $|f''|$ can occur on any subinterval.  It may be fine if
$f''(x)=\sin(1/x)$ on $[0,1]$ but is a poor estimate for $f(x)=\sqrt{x}$.
If $f''$ was positive and increasing then 
$\norm{f''\chi_{[x_{i-1}, x_i]}}_\infty = f''(x_i)<\norm{f''}_\infty$.
This estimate can then be used on each subinterval.  Similarly if
$f$ is decreasing.\\

\refstepcounter{projects}
\label{unbounded}
\noindent
{\bf
\theprojects.
Unbounded integrands.}  It is not necessary for $f'$ or $f''$ to be 
integrable.  If not, we may be able to integrate against a polynomial
with a zero of sufficient multiplicity.  For example, suppose 
$f\in C^2((0,1])$ such that
$\int_0^1 f(x)\,dx$ exists and as $x\to 0^+$ we have
$f(x)=o(1/x)$ and 
$f'(x)=o(1/x^2)$.
An example of such a function on $[0,1/2]$
is $f(x)=|\log x|^\alpha$ for each real $\alpha$.
Let $p(x)=x^2$. 
Then $\int_0^1f''(x)p(x)\,dx=f'(1)-2f(1)+2\int_0^1f(x)\,dx$. 
(This is Taylor's theorem.)  Show this
leads to a quadrature formula with error a multiple of 
$|\int_0^1f''(x)x^2\,dx|$.
If also $f''(x)=O(1/x^2)$ as $x\to 0^+$ then this integral is bounded by
$\sup_{x\in[0,1]}|f''(x)x^2|$.
There are similar results when $f(x)\sim c_1/x$ for some constant $c_1$ and
$f'(x)\sim c_2/x^2$ for some constant $c_2$.  It is easy to modify this
for higher order singularities.\\

\refstepcounter{projects}
\label{hk}
\noindent
{\bf
\theprojects.
The Henstock--Kurzweil integral.}  The error estimates all depend
on existence of $\intab f''(x)p(x)\,dx$.  There are functions that
are differentiable at each point  for which the derivative is not integrable
in the Riemann or Lebesgue sense.  An example is given by taking
$g\fn[0,1]\to\R$ as $g(x)=x^2\sin({x^{-3}})$ for $x>0$ and $g(0)=0$.
Then $g'$ exists at each point of $[0,1]$ but is not continuous at $0$.
Since the derivative is not bounded, $\int_0^1g'(x)\,dx$ does not
exist as a Riemann integral.  Since $\int_0^1|g'(x)|\,dx=\infty$, we
have $g'\notin L^1([0,1])$.  In this case, $\int_0^1g'(x)\,dx$ exists
as an improper Riemann integral.  However, a construction in 
\cite{jeffery} shows
how to use a Cantor set to piece together such functions so that improper Riemann integrals
do not exist but the Henstock--Kurzweil integral exists.

The Henstock--Kurzweil integral is defined in terms
of Riemann sums that are chosen somewhat more carefully than in 
Riemann integration.  It has the property that if $g'$ exists then
$\int_a^bg'(x)\,dx=g(b)-g(a)$.  In fact, if $g$ is continuous,
this fundamental theorem
of calculus formula will still hold when $g'$ fails to exist on countable
sets and certain
sets of measure zero.  See \cite{gordon}.  Conditionally convergent
integrals such as $\int_0^\infty x^2\sin(e^x)\,dx$ also exist in this
sense.
With the Henstock--Kurzweil integral there is the estimate
$|\intab f(x)g(x)\,dx|\leq \norm{f}\norm{g}_{\bv}$.  The
Alexiewicz norm of $f$ is 
$\norm{f}=\sup_{[c,d]\subset[a,b]}|\int_c^d f(x)\,dx|$.
The function $g$ must be of bounded variation and $\norm{g}_{\bv}=
\norm{g}_\infty+ Vg$, where $Vg$ is the variation of $g$.  
See \cite{lee}.

The conditions on $f$ can then be relaxed to $f''$ integrable in 
the Henstock--Kurzweil sense and we can estimate $\intab f''(x)p(x)\,dx$
using the Alexiewicz norm $\norm{f''}$.  See \cite{dingyeyang}.  
In fact, $f''$ need not
even be a function.  The same estimates hold when $f'$ is merely continuous
and then $f''$ exists in the distributional sense.  See \cite{talviladenjoy}.
Similarly if $f'$ has jump discontinuities of finite magnitude.  
See \cite{talvilaregulated}.

\end{document}